\documentclass[12pt, reqno ]{amsart}
\usepackage{amsmath,amssymb,amsbsy,amsfonts,amsthm,latexsym, amsopn,amstext,amsxtra,euscript,amscd,mathrsfs,color,bm,multirow,tabularx, enumitem, xcolor,url,todonotes,dirtytalk,comment,lipsum}
  
\usepackage[colorlinks,linkcolor=blue,anchorcolor=blue,citecolor=blue]{hyperref}
\usepackage[noadjust]{cite}

\DeclareMathOperator{\Mod}{mod}

\begin{document}
\newtheorem{problem}{Problem}[section]
\newtheorem{theorem}{Theorem}[section]
\newtheorem{lemma}[theorem]{Lemma}
\newtheorem{prop}[theorem]{Proposition}

\theoremstyle{definition}
\newtheorem{crit}[theorem]{Criterion}
\newtheorem{claim}[theorem]{Claim}
\newtheorem{cor}[theorem]{Corollary}
\newtheorem{definition}{Definition}
\newtheorem{question}[theorem]{Question}
\newtheorem{rem}[theorem]{Remark}
\newtheorem{Note}[theorem]{Notation}

%%%%%%%%%%%%%%%%%%%%%%%%%
% Alphabet calligraphic %
%%%%%%%%%%%%%%%%%%%%%%%%%

\def\cA{{\mathcal A}}
\def\cB{{\mathcal B}}
\def\cC{{\mathcal C}}
\def\cD{{\mathcal D}}
\def\cE{{\mathcal E}}
\def\cF{{\mathcal F}}
\def\cG{{\mathcal G}}
\def\cH{{\mathcal H}}
\def\cI{{\mathcal I}}
\def\cJ{{\mathcal J}}
\def\cK{{\mathcal K}}
\def\cL{{\mathcal L}}
\def\cM{{\mathcal M}}
\def\cN{{\mathcal N}}
\def\cO{{\mathcal O}}
\def\cP{{\mathcal P}}
\def\cQ{{\mathcal Q}}
\def\cR{{\mathcal R}}
\def\cS{{\mathcal S}}
\def\cT{{\mathcal T}}
\def\cU{{\mathcal U}}
\def\cV{{\mathcal V}}
\def\cW{{\mathcal W}}
\def\cX{{\mathcal X}}
\def\cY{{\mathcal Y}}
\def\cZ{{\mathcal Z}}

%%%%%%%%%%%%%%%%%%%%%%%
% Alphabet blackboard %
%%%%%%%%%%%%%%%%%%%%%%%
\def\A{{\mathbb A}}
\def\B{{\mathbb B}}
\def\C{{\mathbb C}}
\def\D{{\mathbb D}}
\def\E{{\mathbb E}}
\def\F{{\mathbb F}}
\def\G{{\mathbb G}}
\def\I{{\mathbb I}}
\def\J{{\mathbb J}}
\def\K{{\mathbb K}}
\def\L{{\mathbb L}}
\def\M{{\mathbb M}}
\def\N{{\mathbb N}}
\def\O{{\mathbb O}}
\def\P{{\mathbb P}}
\def\Q{{\mathbb Q}}
\def\R{{\mathbb R}}
\def\S{{\mathbb S}}
\def\T{{\mathbb T}}
\def\U{{\mathbb U}}
\def\V{{\mathbb V}}
\def\W{{\mathbb W}}
\def\X{{\mathbb X}}
\def\Y{{\mathbb Y}}
\def\Z{{\mathbb Z}}

\def\ep{{\mathbf{e}}_p}
\def\eq{{\mathbf{e}}_q}
\def\cal#1{\mathcal{#1}}

\def\scr{\scriptstyle}
\def\\{\cr}
\def\({\left(}
\def\){\right)}
\def\[{\left[}
\def\]{\right]}
\def\<{\langle}
\def\>{\rangle}
\def\fl#1{\left\lfloor#1\right\rfloor}
\def\rf#1{\left\lceil#1\right\rceil}
\def\le{\leqslant}
\def\ge{\geqslant}
\def\eps{\varepsilon}
\def\mand{\qquad\mbox{and}\qquad}

\def\sssum{\mathop{\sum\ \sum\ \sum}}
\def\ssum{\mathop{\sum\, \sum}}
\def\ssumw{\mathop{\sum\qquad \sum}}

\def\vec#1{\mathbf{#1}}
\def\inv#1{\overline{#1}}
\def\num#1{\mathrm{num}(#1)}
\def\dist{\mathrm{dist}}

\def\fA{{\mathfrak A}}
\def\fB{{\mathfrak B}}
\def\fC{{\mathfrak C}}
\def\fU{{\mathfrak U}}
\def\fV{{\mathfrak V}}

\newcommand{\bflambda}{{\boldsymbol{\lambda}}}
\newcommand{\bfxi}{{\boldsymbol{\xi}}}
\newcommand{\bfrho}{{\boldsymbol{\rho}}}
\newcommand{\bfnu}{{\boldsymbol{\nu}}}

\def\GL{\mathrm{GL}}
\def\SL{\mathrm{SL}}

\def\Hba{\overline{\cH}_{a,m}}
\def\Hta{\widetilde{\cH}_{a,m}}
\def\Hb1{\overline{\cH}_{m}}
\def\Ht1{\widetilde{\cH}_{m}}

\def\flp#1{{\left\langle#1\right\rangle}_p}
\def\flm#1{{\left\langle#1\right\rangle}_m}

\def\Zm{\Z/m\Z}

\def\Err{{\mathbf{E}}}
\def\O{\mathcal{O}}

\def\cc#1{\textcolor{red}{#1}}
\newcommand{\fix}[1]{{\color{red} #1}}

\newcommand{\comm}[1]{\marginpar{%
\vskip-\baselineskip %raise the marginpar a bit
\raggedright\footnotesize
\itshape\hrule\smallskip#1\par\smallskip\hrule}}
\newcolumntype{L}{>{\raggedright\arraybackslash}X}

\def\xxx{\vskip5pt\hrule\vskip5pt}

\def\dmod#1{\,\left(\textnormal{mod }{#1}\right)}

%%%%%%%%%%%%%%%%%%
%% PAPER BEGINS %%
%%%%%%%%%%%%%%%%%%

\title{Norm-Euclideanity in cyclic fields}

\author[G.K. Bagger]{Gustav Kj\ae rbye Bagger}
\address{School of Science, The University of New South Wales Canberra, Australia}
\email{g.bagger@unsw.edu.au}
 
\date{\today}
\pagenumbering{arabic}

%\keywords{??}
%\subjclass[2010]{??}
\let\thefootnote\relax
\footnote{\textit{Affiliation}: School of Science, The University of New South Wales Canberra, Australia.}
\footnote{\textit{Key phrases}: norm-Euclideanity, character sums, least non-residues}
\begin{abstract}
Consider a cyclic number field of prime degree $p$. For a fixed degree, there are only finitely many such fields which are norm-Euclidean. When $p=2$ or $3$, these have been completely determined. If instead $p>3$, one can bound the discriminant of such fields from above. We substantially sharpen these upper bounds via an explicit Burgess-type criterion and describe a method for completely resolving other small prime cases.
\end{abstract}

\maketitle
\section{Introduction}
Let $p\geq 3$ be prime and let $K$ be a cyclic (number) field of degree $p$ over $\Q$. Let $\cO_K$, $f$ and $\Delta$ denote the ring of integers, conductor and discriminant of $K$, respectively. By the conductor-discriminant formula, $\Delta=f^{p-1}$. Call $K$ \emph{norm-Euclidean} if, for every $\alpha,\beta \in \cO_K$ with $\beta\neq 0$, there exists some $\gamma \in \cO_K$ such that $\vert N(\alpha-\gamma\beta)\vert < \vert N(\beta) \vert$, where $N=N_{K/\Q}$ denotes the absolute norm map. Since it is necessary for a norm-Euclidean number field to be a unique factorisation domain, we assume that $h_K=1$, the class number of $K$. By genus theory, $f$ must be either a prime congruent to $ 1 \Mod{p}$ or $f=p^2$. For ease of notation, we will write $L= \ln f$.

For a fixed prime $p\geq 3$, Heilbronn proved that there are only finitely many norm-Euclidean Galois fields of degree $p$, although this result is inexplicit. However, in the setting of cyclic quadratic and cyclic cubic fields, the norm-Euclidean fields have been completely determined~\cite{ChatlandDavenport1950},~\cite{BaggerNorm2026}. McGown~\cite{McGown2012} gave the first explicit bound on the conductor and discriminant of any norm-Euclidean cyclic field of prime degree. This was subsequently improved by Lezowski and McGown.
\begin{prop} \label{P: Lez and McG}
Consider a cyclic field $K$ of degree $p> 3$ and fix an integer $2 \leq r \leq 10$. There exists an effectively computable constant $D_\mathrm{LM}(r)>0$ such that if
$$ D_\mathrm{LM}(r)(p-1)^rL^{5/2} < f^{\frac{r-1}{4r}}, $$
then $K$ is not norm-Euclidean.
\end{prop}
This is a reformulation of~\cite[Proposition 2.4]{LezowskiMcGown2017}. We improve on this result, both in terms of the leading constant and by a $L^{1/2}$ saving.
\begin{theorem} \label{T: main}
Consider a cyclic field $K$ of degree $p>3$ and fix an integer $2 \leq r< p$. Assume $2p < f^{1/2r}$, $f\geq 10^{40}$ and set $D_p(f,r) = 3.311(2r-1 +r!)^{1/2}(1-2pf^{-1/2r})^{-r}$.
If
\begin{equation} \label{Eq: main}
D_p(f,r)(p-1)^rL^{2}  < f^{\frac{r-1}{4r}},
\end{equation}
then $K$ is not norm-Euclidean.
\end{theorem}
Note that the technical restriction $f\geq 10^{40}$ does not impede us: the bound in Proposition~\ref{P: Lez and McG} is minimal at $p=5$ where $f\geq 10^{55}$. By applying Theorem~\ref{T: main} for small primes $p>3$, we obtain a necessary condition for norm-Euclideanity solely dependant on the conductor $f$.
\begin{prop} \label{P: comp-table}
Let $K$ be a norm-Euclidean cyclic field of prime degree $p$ where $3<p<100$. Then its conductor $f$ must be bounded above by $f_0$ given in Table~\ref{Tab: conductors}.
\end{prop}
\begin{table}[!h]
    \centering
    
    \begin{tabular}{c|c||c|c||c|c}
        \multicolumn{1}{c}{$p$} & \multicolumn{1}{c||}{$f_0$} &
        \multicolumn{1}{c}{$p$} & \multicolumn{1}{c||}{$f_0$} &
        \multicolumn{1}{c}{$p$} & \multicolumn{1}{c}{$f_0$} \\
        \hline
        5  & $4.2 \times 10^{40}$ & 37 & $8.7 \times 10^{59}$ & 73 & $7.0 \times 10^{65}$ \\
        7  & $2.3 \times 10^{44}$ & 41 & $6.9 \times 10^{60}$ & 79 & $3.4 \times 10^{66}$ \\
        11 & $7.3 \times 10^{48}$ & 43 & $1.8 \times 10^{61}$ & 83 & $8.9 \times 10^{66}$ \\
        13 & $2.9 \times 10^{50}$ & 47 & $1.1 \times 10^{62}$ & 89 & $3.6 \times 10^{67}$ \\
        17 & $9.0 \times 10^{52}$ & 53 & $1.2 \times 10^{63}$ & 97 & $2.0 \times 10^{68}$ \\
        19 & $9.4 \times 10^{53}$ & 59 & $1.1 \times 10^{64}$ & & \\
        23 & $5.1 \times 10^{55}$ & 61 & $2.0 \times 10^{64}$ & & \\
        29 & $6.1 \times 10^{57}$ & 67 & $1.3 \times 10^{65}$ & & \\
        31 & $2.4 \times 10^{58}$ & 71 & $4.1 \times 10^{65}$ & & \\
    \end{tabular}
    
    \caption{Conductor bounds for norm-Euclidean cyclic fields of degree $p$}
    \label{Tab: conductors}
\end{table}
\vspace{-5 pt}
These bounds improve substantially upon previous conductor bounds: For each prime $p$, we obtain an upper bound roughly 15 orders of magnitude smaller than~\cite{LezowskiMcGown2017}. We give a simple version of Theorem~\ref{T: main}, which might prove useful for applications.
\begin{prop} \label{P: main-simp}
Consider a cyclic field $K$ of prime degree $p>3$. If
\begin{equation*}
2.177 (p-1)  < f^{1/16}L^{-1},
\end{equation*}
then $K$ is not norm-Euclidean.
\end{prop}

Define the quantities $W(f,h,r) := 2r-1 + r!f^{1/2}h^{-r}$ and $\lambda_h(v) := \sum_{q \mid v} \lceil h/q \rceil$ where $h$ is a positive integer and the sum is over prime divisors of $v$. Set the empty sum $\lambda_h(1) := 0$. Our method relies on a generalisation of~\cite[Theorem 3.3]{BaggerNorm2026}.
\begin{theorem}\label{T: criterion}
Let $p > 3$ be prime and $\chi$ be a $p$-th order Dirichlet character modulo a prime $f$.
Let $\omega$ denote a non-trivial $p$-th root of unity. Let $r,h,u,v\in\N$ where $u=u_{1}u_{2}\dots u_{k}$ and $v=v_{1}v_{2}\dots v_{\ell}$ are products of distinct primes satisfying $u_{i}<h$, $v_j<f$ and $(u,v)=1$. Let $H \in (0,f)$ be a real number and set $X=H/h$. Suppose $r<p$, $p\lambda_h(v)<h$, $\frac{X}{u}\geq 1$, $X\geq 2$, and $2HX\leq f$, and let
\begin{equation}
E_u(X) = 1 - \frac{\pi^2}{6}\sigma(u) \left( \frac{\sigma(u)}{4} + \frac{\phi(u)}{u} + \frac{\phi(u)}{X} \right)\frac{1}{X},
\end{equation}
where $\sigma(u) = \sum_{d \mid u} d$; for $u=1$, the term $\frac{\phi(u)}{u}$ is replaced by $2$. If $E_u(X) > 0$ and
\begin{equation} \label{Eq: main-criterion}
    \frac{1}{E_u(X)}\frac{\pi^2}{6}\frac{\sigma(u)}{\phi(u)}uhf^{1/2}\left(\frac{(p-1)h}{h-p \lambda_h(v)}\right)^{2r}\frac{W(f,h,r)}{H^2}<1,
\end{equation}
then there exists a positive integer $m \leq H$ satisfying $(m,uv)=1$ and $\chi(m)=\omega$. 
\end{theorem}
The classical version of this result is recovered by replacing $\lambda_h(v) = \ell$, which additionally holds in the setting $p=3$. This follows from $\lceil h/v_j \rceil = 1$ whenever one imposes the assumption $h\leq v_j$. 

The outline of this paper is as follows. In Section~\ref{S: Heilbronn}, we describe Heilbronn's criterion for norm-Euclideanity. By bounding a certain character sum from above and below, we prove Theorem~\ref{T: criterion}. In Section~\ref{S: conductors}, we prove the auxiliary result Theorem~\ref{T: main-ugly}; first by assuming existence of small least non-residues, and then in general. Theorem~\ref{T: main} will be an almost immediate consequence of this. We also describe the method for computing the conductor bound in Table~\ref{Tab: conductors} and prove Proposition~\ref{P: main-simp}. Finally, in Section~\ref{S: future}, we discuss some future directions and potential improvements.

\section{Detecting failure of norm-Euclideanity} \label{S: Heilbronn}
The aim of this section is to motivate and prove Theorem~\ref{T: criterion}. Let $q_1<q_2$ denote the two smallest prime $p$-th non-residues modulo $f$. By class field theory, this is equivalent to the smallest primes inert across the extension $K/\Q$. Whenever the field $K$ is fixed, we will denote by $\chi$ a $p$-th order Dirichlet character modulo a prime $f$. We have the following sufficient condition for failure of norm-Euclideanity.
\begin{prop} \label{P: Heilbronn}
Suppose $(f,p)=1$ and that there exists $m \in \Z^+$ with
$$ (m,q_1q_2)=1 \quad \text{and} \quad \chi(m) = \chi(q_2)^{-1}, $$
such that any of the following conditions hold:
\begin{enumerate}[label=(\arabic*)\quad]
    \item $mq_2k \not\equiv f \Mod q_1^2$, \quad $\forall k\text{ such that } 1\leq k < q_1$, \quad $(q_1-1)(q_2m-1)\leq f$
    \item $q_1\neq 2,3$, \quad $3q_1q_2m\ln q_1 < f$
    \item $q_1\neq 2,3,7$, \quad $2.1 q_1q_2m\ln q_1 < f$
    \item $q_1=2,q_2 \neq 3$, \quad $3q_2m < f$
    \item $q_1=3,q_2 \neq 5$, \quad $5q_2m < f$.
\end{enumerate}
Then $K$ is not norm-Euclidean.
\end{prop}
This is~\cite[Theorem 3.1]{McGown2012} based on Heilbronn's criterion for norm-Euclideanity. Since the case of cyclic cubic fields has been completely resolved, we will concern ourselves with the case when $p>3$. We may assume $f$ is prime and also $f \equiv 1 \Mod{p}$ since this is necessary for norm-Euclideanity, see~\cite[Theorem 4.1]{McGown2012}. We have an additional sufficient condition for failure of norm-Euclideanity.
\begin{prop}[{\unskip\cite[Lemma 5.1]{LezowskiMcGown2017}}] \label{P: Lemma 5.1}
Suppose $(f,p)=1$ and that there exists $m \in \Z^+$ with
$$ (m,q_1q_2)=1 \quad \text{and} \quad \chi(m) = \chi(q_2)^{-1}, $$
such that $q_1 \neq 2,3,7$ and
$$ f\geq \text{max}\left\{q_1,\frac{2.1}{p}f^{1/p}L\right\}q_2m. $$
Then $K$ is not norm-Euclidean.
\end{prop}
Theorem~\ref{T: main} will ultimately follow from invoking Theorem~\ref{T: criterion} where $H$ is chosen as a function of $f$ such that $m$ satisfies one of the conditions in Propositions~\ref{P: Heilbronn} and~\ref{P: Lemma 5.1}.
\subsection{Bounding the character sum from below}
We outline the analogous proof of Theorem~\ref{T: criterion} in the cubic case since it serves as a model for our generalisation. The proof relies in part on bounding the absolute value of the sum $\sum_{n=0}^{h-1}\chi(z+n)$ from below under the assumption that $\chi(z+n)$ rarely equals some fixed cube root of unity $\omega$. Assume also that $\chi(z+n)\neq 0$ for all $0\leq n<h$ and take $\eta$ to be an arbitrary $p$-th root of unity. Let
$$ a_t(z,\eta) := \#\{0\leq n < h: \eta^{-1}\chi(z+n) = \zeta^t\}, $$
where $0 \leq t < p$ and $\zeta = e^{2 \pi i/p}$. Note that $\sum_t a_t(z,\eta) = h$. If $p=3$, we assume $a_0(z,\omega) \leq \ell$, say, and argue as follows. Since $\Re(\zeta)=\Re(\zeta^2) = -\frac{1}{2}$, we have
$$ \left\vert \sum_{n=0}^{h-1}\chi(z+n)\right\vert = \vert a_0(z,\omega)+a_1(z,\omega)\zeta+a_2(z,\omega)\zeta^2\vert \geq \left\vert a_0(z,\omega) - \frac{h-a_0(z,\omega)}{2}\right\vert. $$
Under the assumption that $p\ell = 3\ell \leq h$, we may conclude
$$ \left\vert \sum_{n=0}^{h-1}\chi(z+n)\right\vert \geq \frac{h-3\ell}{2}. $$
Under the same assumptions, one can show that this sum cannot be too large, leading to a contradiction and effectively proving Theorem~\ref{T: criterion} for cyclic cubic fields. 

The main obstruction to generalising this result is the distribution of $p$-th roots of unity. We are only restricting the behaviour of one such root and one might imagine that the other roots could cancel sufficiently to make $\vert \sum_{n=0}^{h-1}\chi(z+n)\vert$ arbitrarily small. The key is to instead consider sums over all conjugates $\chi^j$ simultaneously. For an integer $z$, define the quantity
$$ A_j(z) := \sum_{n=0}^{h-1} \chi(z + n)^j \quad \text{where } 1 \leq j < p. $$
If $z$ satisfies $\chi(z+n) \neq 0$ for every $0\leq n < h$, set $A_0(z) := h$. For each $1\leq j < p$, define the $j$-th conjugate sum
$$ S_j(f,h,r) := \sum_{x\Mod{f}} \vert A_j(x)\vert^{2r}. $$
In this notation, the cubic argument corresponds to bounding the sum $S_1$ from both directions and arriving at a contradiction. We will instead proceed by considering the sum over all non-trivial conjugates $S_p(f,h,r) := \sum_{j = 1}^{p-1}S_j(f,h,r)$.

\begin{lemma}\label{L: r=1} Let $\eta$ be a $p$-th root of unity and fix an integer $z$ such that $\chi(z+n) \neq 0$ for all $0\leq n<h$. Then
$$\sum_{j=1}^{p-1} \vert A_j(z)\vert^2 \geq \frac{(h-pa_0(z,\eta))^2}{p-1}. $$
\end{lemma}
\begin{proof}
For any choice of $\eta$, we have
$$ A_j(z) = \sum_{t=0}^{p-1} a_t(z,\eta) \left(\eta\zeta^t\right)^j = \eta^j\sum_{t=0}^{p-1} a_t(z,\eta) \zeta^{tj}. $$
Then
\begin{align*}
\sum_{j=0}^{p-1} \vert A_j(z)\vert^2 = & \sum_{j=0}^{p-1}\vert\eta^j\vert^2\left\vert\sum_{t=0}^{p-1} a_t(z,\eta) \zeta^{tj}\right\vert^2 \\
= & \sum_{j=0}^{p-1}\left(\sum_{t_1=0}^{p-1} a_{t_1} \zeta^{t_1j}\right)\left(\sum_{t_2=0}^{p-1} a_{t_2} \zeta^{-t_2j}\right) \\
= & \sum_{j=0}^{p-1}\sum_{t_1=0}^{p-1}\sum_{t_2=0}^{p-1} a_{t_1}a_{t_2} \zeta^{j(t_1-t_2)} \\
= & \sum_{t_1=0}^{p-1}\sum_{t_2=0}^{p-1} a_{t_1}a_{t_2}\left( \sum_{j=0}^{p-1} \zeta^{j(t_1-t_2)}\right).
\end{align*}
But orthogonality gives
$$\sum_{j=0}^{p-1} \zeta^{j(t_1-t_2)}=\left\{\begin{array}{ccc}
    p & \vline & t_1=t_2 \\
    0 & \vline & t_1 \neq t_2
\end{array} \right., $$
so we have $\sum_{j=0}^{p-1} \vert A_j(z)\vert^2 = p\sum_{t=0}^{p-1} a_t(z,\eta)^2 $. Applying Cauchy--Schwarz, we have
\begin{align*}
\left(\sum_{t=1}^{p-1} 1^2\right)\cdot\left(\sum_{t=1}^{p-1} a_t(z,\eta)^2\right) & \geq \left(\sum_{t=1}^{p-1} 1\cdot a_t(z,\eta)\right)^2 \\
\iff \sum_{t=1}^{p-1} a_t(z,\eta)^2 & \geq \frac{\left(\sum_{t=1}^{p-1} a_t(z,\eta) \right)^2}{\sum_{t=1}^{p-1} 1}.
\end{align*}
Thus $\sum_{t=1}^{p-1} a_t(z,\eta)^2 \geq \frac{\left(h-a_0(z,\eta) \right)^2}{p-1}$. Since $\vert A_0 \vert^2 = \vert \sum_{n=0}^{h-1} 1 \vert^2 = h^2$, we may conclude that
\begin{align*}
\sum_{j=1}^{p-1} \vert A_j(z)\vert^2 & = -\vert A_0 \vert^2 + p\sum_{t=0}^{p-1} a_t(z,\eta)^2 \\
& = -h^2 + pa_0(z,\eta)^2 + p \sum_{t=1}^{p-1} a_t(z,\eta)^2 \\
& \geq -h^2 + pa_0(z,\eta)^2 + p \frac{(h-a_0(z,\eta))^2}{p-1} \\
& = \frac{(h-pa_0(z,\eta))^2}{p-1}
\end{align*}
\end{proof}
Lemma~\ref{L: r=1} corresponds to the cubic bound 
$$\frac{1}{2}\sum_{j=1}^2\vert A_j(z)\vert^{2} \geq \left(\frac{h-p\ell}{2}\right)^{2} $$ 
since the pointwise bounds give $\vert A_1(z)\vert, \vert A_2(z)\vert \geq (h-p\ell)/2$. From the sum over $\vert A_j(z)\vert^2$, we may directly infer information about sums over higher powers $\vert A_j(z)\vert^{2r}$.  
\begin{lemma}\label{L: average A_j} For $r\in \Z^+$, $\eta$ a $p$-th root of unity and fixed integer $z$ satisfying $\chi(z+n) \neq 0$ for all $0\leq n < h$, we have
$$\sum_{j=1}^{p-1} \vert A_j(z)\vert^{2r} \geq \frac{(h-pa_0(z,\eta))^{2r}}{(p-1)^{2r-1}}. $$
\end{lemma}
\begin{proof}
Considering the power mean inequality and applying Lemma~\ref{L: r=1}, we have
$$ \left(\frac{1}{p-1}\sum_{j=1}^{p-1} \vert A_j(z)\vert^{2r} \right)^{1/r} \geq \frac{1}{p-1}\sum_{j=1}^{p-1} \vert A_j(z)\vert^2  \geq \left(\frac{h-pa_0(z,\eta)}{p-1}\right)^2. $$
Taking $r$-th powers and rearranging gives the desired bound.
\end{proof}
For $t,q \in \Z$ with $0\leq t<q\leq X$ and $(t,q)=1$, we define the intervals
$$ \mathcal{I}(q,t):=\left(\frac{tf}{q},\,\frac{tf+H}{q}-h+1\right]\, \quad \text{and} \quad
\mathcal{J}(q,t):=\left[\frac{tf-H}{q},\,\frac{tf}{q}-h+1\right)\,. $$
\begin{lemma} \label{L: Intervals}
Suppose $0 \leq t < q \leq X$ with $(t,q) = 1$. Then 
\begin{enumerate}
\item 
Let $0\leq n\leq h-1$.
If $z\in\mathcal{I}(q,t)$ then $q(z+n)-ft\in(0,H]$,
and if $z\in\mathcal{J}(q,t)$ then $q(z+n)-ft\in[-H,0)$.
\item
If $X\geq 2$ and $2HX \leq f$, then the collection
$$\{\mathcal{I}(q,t),\mathcal{J}(q,t): 0\leq t < q\leq X,\ (t,q)=1\}$$
is pairwise disjoint, contained in $[-H,f-H)$.
\end{enumerate}
\end{lemma}
This is a standard result on Burgess-type intervals; see for example McGown and Trudgian~\cite[Proposition 3]{McGownTrudgian2020}. For integers $a,b$, we use the notation $a \parallel b$ to signify $a$ dividing $b$ but not $\frac{b}{a}$.

\begin{prop}\label{P: S Lower bound}
Retain the notation and hypotheses of Theorem~\ref{T: criterion}. For a fixed choice of $p$-th root of unity $\omega$, assume that $\chi(N)\neq\omega$ for all integers $N\in[1,H]$ with $(N,uv)=1$. Then 
$$ S_p(f,h,r) \geq N_u(X)\frac{(h-p\lambda_h(v))^{2r}}{(p-1)^{2r-1}}, $$
where 
$$N_u(X) = \sum_{\substack{1 \leq q\leq X\\u\parallel q}} \sum_{\substack{0 \leq t < q\\(t,q)=1}} \sum_{z \in \mathcal{I}(q,t) \sqcup \mathcal{J}(q,t)}1.$$
\end{prop}
\begin{proof}
Fix integers $0\leq t< q \leq X$ satisfying $(t,q)=1$ and $u\parallel q$. We consider only $z\in\mathcal{I}(q,t)$; for $z\in\mathcal{J}(q,t)$ the same argument applies with $q(z+n)-tf$ replaced by its absolute value, which does not change the character value since $\chi(-1)=1$. This follows from the fact that $\chi(-1)^2 = \chi(-1)^p = 1$ and the primality of $p$. By Lemma \ref{L: Intervals}, if $0\leq n\leq h-1$, then one has $q(z+n)-tf \in (0,H]$. If additionally, $(q(z+n)-tf, uv) = 1$, then $\chi(q(z+n)-tf) \neq \omega$ by assumption. Let us show that $(q(z+n)-tf, uv) \neq 1$ for at most $\lambda_h(v)$ choices of $n$.

If $u_i\mid q(z+n)-tf$ for some $i$, then, since $u_i\mid u\mid q$, we get $u_i\mid tf$; as $(t,q)=1$ we have $u_i\nmid t$, and $u_i\nmid f$ since $f$ is prime and $u_i<h\leq H/2<f$. Hence $(q(z+n)-tf,u)=1$. Thus, it is sufficient to show that for every $1 \leq j \leq \ell$, $v_{j}\mid q(z+n)-tf$ for at most $\lceil h/v_j \rceil$ choices of $0\leq n <h$.

If $v_j\mid q$, then the same argument as above for $u_i$ shows that $v_j\nmid q(z+n)-tf$ for all choices of $n$. If instead $v_j\nmid q$, consider two values $0\leq n_1,n_2\leq h-1$ such that $v_j \mid q(z+n_1)-tf,q(z+n_2)-tf$. Then $v_j \mid q(n_1-n_2)$ and, by assumption, $n_1 \equiv n_2 \Mod{v_j}$. Since every residue class modulo $v_j$ occurs at most $\lceil h/v_j \rceil$ times amongst $0\leq n \leq h-1$, we have $v_j \mid q(z+n)-tf$ for at most $\lceil h/v_j \rceil$ choices of $n$. By considering all prime divisors $v_j \mid v$, we obtain
$$ \#\left\{0\leq n \leq h-1: (v,q(z+n)-tf) \neq 1\right\} \leq \sum_{j=1}^{\ell} \left\lceil\frac{h}{v_j}\right\rceil = \lambda_h(v).$$

Since $q(z+n)-tf\equiv q(z+n)\pmod{f}$ and $0<q\leq X<f$, we have $\chi(q(z+n)-tf)=\chi(q)\chi(z+n)$, where $\chi(q)$ is a $p$-th root of unity. 
The argument above thus implies that $\chi(z+n)$, $0 \leq n \leq h-1$, coincides with the $p$-th root of unity $\omega_q:=\chi(q)^{-1}\omega$ for at most $\lambda_h(v)$ values of $n$. This is equivalent to $a_0(z,\omega_q)\leq \lambda_h(v)$. Moreover, $\chi(z+n) \neq 0$: writing $q(z+n)=tf+N$ with $0<|N|\leq H<f$, we see that $f\mid z+n$ would force $f\mid N$. Since $X\geq 2$ and $2HX\leq f$, Lemma~\ref{L: Intervals} shows that the intervals $\mathcal{I}(q,t)$, $\mathcal{J}(q,t)$ with $0\leq t<q\leq X$, $(t,q)=1$, are pairwise disjoint subsets of $[-H,f-H)$. The integers $z$ appearing below are therefore pairwise distinct and lie in an interval of length $f$, so by the periodicity of $\chi^j$, each of them contributes at most one term to $S_j(f,h,r)$. Then
$$S_p(f,h,r) = \sum_{j = 1}^{p-1}\sum_{x\Mod{f}}\vert A_j(x)\vert^{2r} \geq \sum_{\substack{1 \leq q\leq X\\u\parallel q}} \sum_{\substack{0 \leq t < q\\(t,q)=1}} \sum_{z \in \mathcal{I}(q,t) \sqcup \mathcal{J}(q,t)} \sum_{j = 1}^{p-1}\vert A_j(z)\vert^{2r}. $$
Since $pa_0(z,\omega_q) \leq p\lambda_h(v) < h$ by assumption, we have $h-pa_0(z,\omega_q) \geq h-p\lambda_h(v)$. Applying Lemma~\ref{L: average A_j}, we obtain
\begin{align*}
S_p(f,h,r) & \geq \sum_{\substack{1 \leq q\leq X\\u\parallel q}} \sum_{\substack{0 \leq t < q\\(t,q)=1}} \sum_{z \in \mathcal{I}(q,t) \sqcup \mathcal{J}(q,t)} \sum_{j = 1}^{p-1}\vert A_j(z)\vert^{2r} \\
& \geq \sum_{\substack{1 \leq q\leq X\\u\parallel q}} \sum_{\substack{0 \leq t < q\\(t,q)=1}} \sum_{z \in \mathcal{I}(q,t) \sqcup \mathcal{J}(q,t)} \frac{(h-p\lambda_h(v))^{2r}}{(p-1)^{2r-1}} \\
& = N_u(X)\frac{(h-p\lambda_h(v))^{2r}}{(p-1)^{2r-1}},
\end{align*}
as required.
\end{proof}
Via a lattice-point counting argument, one obtains a lower bound for the quantity $N_u(X)$. For a proof of this fact, see for example~\cite[Lemma 3.7]{BaggerNorm2026}.
\begin{lemma}\label{L: N_u}
\begin{equation} \label{Eq: N_u}
N_u(X)\geq E_u(X)\frac{6}{\pi^2}\frac{1}{u}\frac{\phi(u)}{\sigma(u)}X^2h.
\end{equation}
\end{lemma}
\subsection{A Weil-type bound for higher moments} Define the quantity 
$$W'(f,h,r) := 2r-1 + c_p(h,r)f^{1/2}h^{-2r},$$
where
$$ c_p(h,r) := \sum_{d=0}^{\left\lfloor\frac{r}{p}\right\rfloor}\left(\frac{r!}{d!(p!)^d}\right)^2\frac{h^{r-(p-2)d}}{(r-pd)!}.$$
\begin{prop}\label{P: S Upper bound}
Retain the notation and hypotheses of Theorem~\ref{T: criterion}. Then
$$ S_p(f,h,r) \leq (p-1)f^{1/2}h^{2r}W'(f,h,r). $$
\end{prop}
\begin{proof}
Since $p$ is prime, $\chi^j$ is also a character of order $p$ modulo $f$. Thus, independently of $j$,~\cite[Remark 2.1]{Trevino2015} gives the bound 
$$ S_j(f,h,r) \leq f^{1/2}h^{2r}W'(f,h,r). $$
Summing over all $1\leq j < p$ gives the desired bound.
\end{proof}
For $r<p$, the sum $c_p(h,r)$ collapses and we obtain a much simpler expression. 
\begin{lemma}
For any $1\leq r < p$, we have
$$W'(f,h,r) = W(f,h,r). $$
\end{lemma}
\begin{proof}
Since $r<p$, we have $\lfloor \frac{r}{p}\rfloor= 0$ so $$ c_p(h,r) = \left(\frac{r!}{0!(p!)^0}\right)^2\frac{h^{r-(p-2)0}}{(r-p0)!} = (r!)^2 \frac{h^r}{r!}=r!h^r. $$
We conclude
$$W'(f,h,r) = 2r-1 + \frac{r!h^rf^{1/2}}{h^{2r}}=2r-1 + \frac{r!f^{1/2}}{h^r},$$
as required.
\end{proof}
We are now in a position to prove our main result.
\begin{proof}[Proof of Theorem~\ref{T: criterion}]
Assume, for a contradiction, that no such value $m\leq H$ exists. Then, by combining the bounds for $S_p(f,h,r)$ in Propositions~\ref{P: S Lower bound} and~\ref{P: S Upper bound} with the bound~\eqref{Eq: N_u}, we have the inequality
$$ E_u(X)\frac{6}{\pi^2}\frac{1}{u}\frac{\phi(u)}{\sigma(u)}X^2h\frac{(h-p\lambda_h(v))^{2r}}{(p-1)^{2r-1}} \leq S_p(f,h,r) \leq (p-1)f^{1/2}h^{2r}W(f,h,r). $$
By assumption, $E_u(X)>0$ and $h>p\lambda_h(v)$ so the lower bound on $S_p(f,h,r)$ is positive. Rearranging and recalling $X = H/h$, we have
$$ 1 \leq \frac{1}{E_u(X)}\frac{\pi^2}{6}\frac{\sigma(u)}{\phi(u)}uhf^{1/2}\left(\frac{(p-1)h}{h-p \lambda_h(v)}\right)^{2r}\frac{W(f,h,r)}{H^2}. $$
This violates the criterion~\eqref{Eq: main-criterion}, and we conclude there exists a positive integer $m \leq H$ satisfying $(m,uv)=1$ and $\chi(m)=\omega$.
\end{proof}
The twisted version of $E_u(X)$ whenever $u=1$ comes from the non-vanishing of the sum $\sum_{d\mid u}\mu(d)$, explored in more detail following Lemma 3.6 in~\cite{BaggerNorm2026}.

\section{Conductor bounds} \label{S: conductors}
For convenience, we will assume throughout that $f\geq 10^{40}$. In this case, we have 
$$q_1 \leq 1.244 f^{1/4}L \quad \text{and} \quad q_2 \leq 1.670f^{1/4}L^{3/2}.$$
These explicit bounds are given in~\cite[Corollary 2]{MaMcGownRhodesWanner} and can be reduced slightly if one is willing to strengthen the assumption on the size of $f$. If we instead bound $q_1q_2$, we have the following.
\begin{prop} \label{P: q1q2}
Suppose $f\geq 10^{40}$ and $\chi$ has odd order. Then
$$ q_1q_2 < 2.581 f^{1/2}L^2. $$ 
\end{prop}
\begin{proof}
This is effectively~\cite[Corollary 7.5]{LezowskiMcGown2017} followed by applying Proposition 7.3 from the same paper when $f\geq 10^{40}$.
\end{proof}    
Theorem~\ref{T: main} will follow from the auxiliary result below.
\begin{theorem} \label{T: main-ugly}
Consider a cyclic field $K$ of degree $p>3$ and fix an integer $2 \leq r< p$. Set $h = \left\lceil f^{1/{2r}}\right\rceil$, $H = f^{1/2}/(2.581L^2)$, $X = H/h$ and
$$C_p(f,h,r) = 2.581\left(\frac{1}{E_1(X)}\frac{\pi^2}{6}h\frac{W(f,h,r)}{f^{1/2r}}\right)^{1/2}\left(\frac{h}{h-2p}\right)^r. $$
If $f\geq 10^{40}$, $2p<h$ and
\begin{equation} \label{Eq: main-ugly}
C_p(f,h,r)(p-1)^rL^{2}  < f^{\frac{r-1}{4r}},
\end{equation}
then $K$ is not norm-Euclidean.
\end{theorem}
\subsection{Eliminating small non-residues}
In order to invoke Proposition~\ref{P: Lemma 5.1}, we need to establish Theorem~\ref{T: main-ugly} in the case when $q_1 = 2,3,7$.
\begin{prop} \label{P: q1>100}
If $f\geq 10^{40}$, $q_1\leq 100$ and the hypothesis of Theorem~\ref{T: main-ugly} holds, then $K$ is not norm-Euclidean.
\end{prop}
\begin{proof}
Assume first $q_1=2$ and $q_2 = 3$. Then, 
\begin{align*}
72(p-1)f^{1/2}\ln(4f)+35 \leq 72 (f^{1/4}/2)f^{1/2}\ln(4f) + 35 =  36 f^{3/4}\ln(4f) + 35.
\end{align*}
Since $36 f^{3/4}\ln(4f) + 35\leq f$ whenever $f\geq 10^{40}$, $K$ is not norm-Euclidean by~\cite[Proposition 5.1]{McGown2012}. Similarly, one readily establishes that the other exceptional pair $q_1=3$, $q_2=5$ leads to $K$ failing to be norm-Euclidean by verifying 
$$ 507(p-1)f^{1/2}\ln(9f) + 448 \leq f, $$
and appealing to~\cite[Proposition 5.1]{McGown2012} again. By the proof of~\cite[Theorem 5.8]{McGown2012}, it is sufficient to show 
$$ C_p(f,h,r)(p-1)^rL^2<f^{\frac{r-1}{4r}} \quad \text{implies} \quad E'_0(2)(p-1)^2L^{3/2}<f^{3/8} $$ 
where $E'_0(2)< 5.268\times 10^9$. If $h = \left\lceil f^{1/2r}\right\rceil$, then $W(f,h,r) \geq 2r-1$ and $h/f^{1/2r} \geq 1$. Moreover, we have $0<E_1(X)<1$ and $h/(h-2p) > 1$. In particular, 
$$C_p(f,h,r)\geq 2.581 \frac{\pi}{\sqrt{6}}\sqrt{2r-1} =:B_r. $$ 
Then 
\begin{align*}
C_p(f,h,r)(p-1)^rL^2 < f^{\frac{r-1}{4r}} \quad \text{implies} \quad (p-1)^2 < \left(\frac{f^{\frac{r-1}{4r}}}{B_r\cdot L^2}\right)^{2/r}.
\end{align*}
Let $\delta_r := \frac{3}{8} - \frac{r-1}{2r^2} \geq \frac{1}{4}$, which is increasing in $r$. Then
\begin{align*}
    E'_0(2)(p-1)^2L^{3/2}f^{-3/8} & < 5.268\times 10^9 \left(\frac{f^{\frac{r-1}{4r}}}{B_r\cdot L^2}\right)^{2/r} L^{3/2} f^{-3/8} \\
    & = 5.268\times 10^9 B_r^{-2/r} L^{(3r-8)/2r} f^{-\delta_r}.
\end{align*}
For $f\geq 10^{40}$ and $2\leq r \leq 8$, we directly verify that the right hand side is $<1$. When $r\geq 9$, we bound $\delta_r \geq \delta_9 = 3/8 - 4/81 > 0.3256$. For convenience, we remove $B_r^{-2/r}<1$ and check 
$$ 5.268\times 10^9 L^{(3r-8)/2r}/f^{0.3256} < 1. $$
This establishes the claim.
\end{proof}
\subsection{The full conductor bound} To establish Theorem~\ref{T: main-ugly}, it will be useful to rewrite Condition~\eqref{Eq: main-ugly}. Define the quantities $H_C := f^{1/2}L^{-2}/2.581$ and $X_C := H_C/h$. Assuming $2p<h$, we have 
\begin{align} \label{Eq: main-rewritten}
f^{\frac{r-1}{4r}} & > C_p(f,h,r)(p-1)^rL^{2} \nonumber \\
\iff 1 & > C_p(f,h,r)^2(p-1)^{2r}L^{4}f^{\frac{1-r}{2r}}  \nonumber\\
\iff 1 & > 2.581^2\frac{1}{E_1(X_C)}\frac{\pi^2}{6}h\frac{W(f,h,r)}{f^{1/2r}}\left(\frac{(p-1)h}{h-2p}\right)^{2r}L^4 f^{\frac{1-r}{2r}} \nonumber \\
\iff 1 & > \frac{1}{E_1(X_C)}\frac{\pi^2}{6}hf^{1/2}\left(\frac{(p-1)h}{h-2p}\right)^{2r}\frac{W(f,h,r)}{H_{C}^2}.
\end{align}

\begin{proof}[Proof of Theorem~\ref{T: main-ugly}]
Since we may assume $q_1> 100$ by Proposition~\ref{P: q1>100}, we will ultimately invoke Proposition~\ref{P: Lemma 5.1}. Thus, it is sufficient for failure of norm-Euclideanity to find a positive integer $m$ with $(m,q_1q_2)=1$ and $\chi(m) = \chi(q_2)^{-1}$ satisfying
$$ f\geq \text{max}\left\{q_1,\frac{2.1}{p}f^{1/p}L\right\}q_2m. $$
Set $h = \lceil f^{1/2r}\rceil$ and let
$$ H_1:=\frac{f}{q_1q_2}, \quad H_2:= \frac{pf^{1- 1/p}}{2.1 L q_2} \quad \text{and} \quad H_3 := \sqrt{\frac{fh}{2}}. $$
In addition, set $H = \text{min}\left\{H_1,H_2,H_3\right\}$, where the first two values correspond to the maximum and the latter is to ensure $2XH \leq f$ where $X = H/h$.
 By the upper bounds on $q_1q_2$ and $q_2$, together with $p\geq 5$, we have 
\begin{align*} 
    H \geq \text{min}\left\{\frac{f^{1/2}}{2.581L^2},\frac{pf^{1- 1/p-1/4}}{3.507 L^{5/2}}, \frac{f^{1/2+1/4r}}{\sqrt{2}}\right\} & \geq \frac{f^{1/2}}{L^2} \cdot \text{min}\left\{\frac{1}{2.581},\frac{f^{1/20}}{0.7014 L^{1/2}}\right\} \\
    & \geq \frac{f^{1/2}}{2.581 L^2} = H_C.
\end{align*}
The reduction follows from the inequalities $1< f^{1/2r}$ and $\frac{1}{2.581}<\frac{f^{1/20}}{0.7014 L^{1/2}}$ holding in the range $f\geq 10^{40}$. Since $H_C\leq H$, we also have $X_C\leq X$. Recall that we assume $2p<h$ and we will show $\lambda_h(v)p < h$ for any choice of $v \mid q_1q_2$. Via Criterion~\eqref{Eq: main-rewritten}, we will establish the following:
\begin{align*}
& \frac{1}{E_1(X_C)}\frac{\pi^2}{6}hf^{1/2}\left(\frac{(p-1)h}{h-2p}\right)^{2r}\frac{W(f,h,r)}{H_{C}^2} \\
\geq \ & \frac{1}{E_u(X)}\frac{\pi^2}{6}\frac{\sigma(u)}{\phi(u)}uhf^{1/2}\left(\frac{(p-1)h}{h-p \lambda_h(v)}\right)^{2r}\frac{W(f,h,r)}{H^2}.
\end{align*}
Equivalently
\begin{equation} \label{Eq: proof of main condition}
1 \geq \frac{E_1(X_C)}{E_u(X)}\frac{\sigma(u)}{\phi(u)}u\left(\frac{h-2p}{h-\lambda_h(v)p}\right)^{2r}\left(\frac{H_{C}}{H}\right)^2.
\end{equation}
Let $T=16rp\frac{h}{h-2p}$ and $uv = q_1q_2$. Moreover, let $q_i \mid u$ when $q_i< \text{min}\{h,T\}$. The variable $T$ is chosen to match with the hypothesis of Theorem~\ref{T: main-ugly}. Let $A := h/(h-2p)$ and recall the quantity
$$ B_r = 2.581 \frac{\pi}{\sqrt{6}}\sqrt{2r-1}. $$
From the definition of $C_p(f,h,r)$ and $W(f,h,r)$, we have $C_p(f,h,r) \geq B_rA^r$. Moreover,
$$C_p(f,h,r)(p-1)^rL^2 < f^{\frac{r-1}{4r}} \implies (p-1)A < \frac{f^{\frac{r-1}{4r^2}}}{(L^{2}B_r)^{1/r}} =:P. $$
Now $p\geq 5$, so $pA \leq 5/4(p-1)A$. We have
$$ T = 16rpA \leq 20r(p-1)A < 20rP. $$
Since $100<q_1<q_2$, we also have
\begin{equation} \label{Eq: CaseII-III sig/phi}
\frac{\sigma(u)}{\phi(u)} = \frac{\prod_{q_i\mid u} (q_i+1)}{\prod_{q_i\mid u} (q_i-1)} = \prod_{q_i\mid u} \frac{q_i+1}{q_i-1} \leq \left(\frac{102}{100}\right)^2=1.0404,
\end{equation}
where the last inequality follows from $\frac{x+1}{x-1}=1+\frac{2}{x-1}$ being decreasing in $x>1$. Additionally, we have $\sigma(u) \leq 1.0404\cdot\phi(u) \leq 1.0404u$. In order to invoke Theorem~\ref{T: criterion}, we must also verify $X/u \geq 1$, $X\geq 2$ and $E_u(X)>0$. If we assume $u=1$ and $X/u^2 \geq 5$, then the first two claims follow since $X_C\leq X$. Moreover 
\begin{align} \label{Eq: CaseI E_u(X)}
    E_u(X) = \ & 1 - \frac{\pi^2}{6}\left(\frac{1}{4}+2+\frac{1}{X}\right)\frac{1}{X} \nonumber \\
    \geq \ & 1 - \frac{\pi^2}{6}\left(\frac{1}{4}+2+\frac{1}{5}\right)\frac{1}{5} \geq 0.1939.
\end{align}
If instead we let $u> 1$ and assume $X/u^2 \geq 2$, then the first two claims again follow. Then
\begin{align} \label{Eq: CaseII-III E_u(X)}
E_u(X) \geq & \ 1 -\frac{\pi^2}{6}(1.0404u)\left(\frac{1.0404u}{4} + 1 + \frac{1.0404u}{X}\right)\frac{1}{X} \nonumber \\
\geq & \ 1- \frac{\pi^2}{6}\left(\frac{1.0404^2}{4}\frac{u^2}{X} + \frac{1.0404}{u}\frac{u^2}{X} + \frac{1.0404^2}{X}\frac{u^2}{X}\right) \nonumber \\
\geq & \ 1 - \frac{\pi^2}{6}\left(\frac{1.0404^2}{8} + \frac{1.0404}{2\cdot 101} + \frac{1.0404^2}{4\cdot 101^2}\right) > 0.7689.
\end{align} 
Thus, we need only check that $X\geq 5$ for $u=1$ and $X/u^2 \geq 2$ for $u>1$ in order to invoke Theorem~\ref{T: criterion}. We split into cases depending on factors of $u$.

\noindent
{\bf Case Ia:}  $u=1$ and $q_1\geq h$.
Condition~\eqref{Eq: proof of main condition} becomes
$$ 1 \geq \frac{E_1(X_C)}{E_1(X)}\left(\frac{h-2p}{h-\lambda_h(q_1q_2)p}\right)^{2r}\left(\frac{H_{C}}{H}\right)^2 = \frac{E_1(X_C)}{E_1(X)}\left(\frac{H_{C}}{H}\right)^2. $$
Moreover, we have
\begin{align*}
    X \geq \ & X_C = \frac{f^{1/2}}{2.581L^2}\frac{1}{h} \geq \frac{f^{1/4}}{2.581L^2(1+f^{-1/4})} \geq 456731. 
\end{align*}
The case follows from the previous observation that $H_C\leq H$ and that 
$$ E_1(X) = 1 - \frac{\pi^2}{6}\left(\frac{1}{4} + 2 + \frac{1}{X}\right)\frac{1}{X}<1 $$
is increasing with $X$.

\noindent
{\bf Case Ib:}  $u=1$ and $q_1 < h$. 
We must necessarily have $T\leq q_1 < q_2$. Condition~\eqref{Eq: proof of main condition} becomes 
$$ 1 \geq \frac{E_1(X_C)}{E_1(X)}\left(\frac{h-2p}{h-\lambda_h(q_1q_2)p}\right)^{2r}\left(\frac{H_{C}}{H}\right)^2. $$
Now
$$ \left\lceil\frac{h}{q_1}\right\rceil < \frac{h}{q_1} + 1 \leq \frac{h}{T} + 1 = 1 + \frac{h-2p}{16rp}. $$
We can bound $\lceil h/q_2\rceil$ similarly, by the worst case when $q_2< h$. We conclude
$$ \left(\frac{h-2p}{h-\lambda_h(v)p}\right)^{2r} \leq \left(\frac{h-2p}{h-\left(2+\frac{h-2p}{8rp}\right)p}\right)^{2r} = \left(\frac{h-2p}{h-2p-\frac{h-2p}{8r}}\right)^{2r} = \left(\frac{1}{1-1/8r}\right)^{2r}, $$
a decreasing function in $r\geq 2$. Thus, we attain the maximum at $r=2$ where 
\begin{equation} \label{Eq: CaseI h-2p}
\left(\frac{h-2p}{h-\lambda_h(v)p}\right)^{2r}  < 1.3.
\end{equation} 
We showed previously that $H_C\leq H$ via $q_1q_2<2.581f^{1/2}L^2$. Since $q_1<h$ by assumption, we have $q_1q_2<1.670f^{1/2}L^{3/2}$ using the pointwise bound on $q_2$ alone. In this case, we obtain 
\begin{align*}
    H \geq \ & \text{min}\left\{\frac{f^{1/2}}{1.670 L^{3/2}},\frac{pf^{1- 1/p-1/4}}{3.507 L^{5/2}}, \frac{f^{1/2+1/4r}}{\sqrt{2}}\right\} \\
    \geq \ & \frac{f^{1/2}}{L^{3/2}} \cdot \text{min}\left\{\frac{1}{1.670},\frac{f^{1/20}}{0.7014 L}\right\} \\
    \geq \ & \frac{f^{1/2}}{1.670 L^{3/2}}.
\end{align*}
Thus, we have the ratio
\begin{equation} \label{Eq: CaseI uH_C/H}
H_C/H \leq \frac{f^{1/2}}{2.581 L^2}\left/ \frac{f^{1/2}}{1.670 L^{3/2}}\right. = \frac{1.670}{2.581 L^{1/2}} < 0.06743. 
\end{equation}
Combining Equations~\eqref{Eq: CaseI E_u(X)},~\eqref{Eq: CaseI h-2p} and~\eqref{Eq: CaseI uH_C/H}, we conclude
$$ \frac{E_1(X_C)}{E_1(X)}\left(\frac{h-2p}{h-\lambda_h(q_1q_2)p}\right)^{2r}\left(\frac{H_{C}}{H}\right)^2 < 1\cdot 1.3 \cdot 0.06743^2 < 1. $$
This, together with the fact that $X\geq X_C > 456731$ as before, completes the case. 

\noindent
{\bf Case II:}  $u=q_1$. Thus, we have $q_1< T$ and $q_2\leq 1.670f^{1/4}L^{3/2}$. Moreover, since 
$$L = 2r \ln f^{1/2r} > 2r \ln q_1> 2r \ln (101),$$
we may assume $r<L/(2\ln 101)$. Setting $c_T := 10/\ln 101$ gives the estimates 
\begin{align*}
    T < \ & 20rP < \frac{20}{2\ln{101}}f^{\frac{r-1}{4r^2}}L^{1-2/r} \leq c_Tf^{1/16}L \\
    \frac{T^2}{h} < \ & c_T^{2}\frac{L^{2-4/r}}{B_r^{2/r}f^{1/2r^2}},
\end{align*}
which will be used frequently throughout. We consider each $X_i:=H_i/h$ in turn.
\begin{align*}
X_1/u^2 \geq \ & \frac{f}{u^3q_2 h} > \frac{1}{1.670 c_T^3 }\frac{f^{3/4-1/2r-3\frac{r-1}{4r^2}}}{L^{9/2-6/r}(1+f^{-1/2r})} \\
\geq \ & \frac{1}{1.670c_T^3}\frac{f^{5/16}}{L^{3/2}(1+f^{-1/4})} > 2.105 \times 10^8 \quad (r=2) \\
X_2/u^2 \geq \ & \frac{5f^{4/5}}{2.1 u^2 L q_2 h} > \frac{5}{2.1\cdot1.670c_T^2}\frac{f^{4/5-1/2r-1/4-\frac{r-1}{2r^2}}}{L^{9/2-4/r}(1+f^{-1/2r})} \\
\geq \ & \frac{5}{2.1\cdot1.670c_T^2}\frac{f^{7/40}}{L^{5/2}(1+f^{-1/4})} > 37.30 \quad (r=2) \\
X_3/u^2 \geq \ & \frac{f^{1/2}h^{1/2}}{\sqrt{2}u^2h} > \frac{1}{\sqrt{2}c_T^2}\frac{f^{1/2+1/4r-1/2r-\frac{r-1}{2r^2}}}{L^{2-4/r}(1+f^{-1/2r})} \\
\geq \ & \frac{1}{\sqrt{2}c_T^2}\frac{f^{1/4}}{(1+f^{-1/4})}>1.506 \times 10^9 \quad (r=2).
\end{align*}
Here, the parenthesised $r$-values denote the worst choices $2\leq r<p$, in the sense that they minimise the lower bound. In particular, we have 
\begin{equation*} 
X/u^2 \geq \text{min}\left\{X_1/u^2,X_2/u^2,X_3/u^2\right\} > 37.30.
\end{equation*}
We consider $H_1$ and $H_3$, before handling the more cumbersome $H_2$ separately.
\begin{align*}
u(H_C/H_1)^2 \leq \ & \frac{f}{2.581^2L^4}\frac{q_1}{f^2/(q_1^2q_2^2)} = \frac{1}{2.581^2}\frac{q_1^3q_2^2}{fL^4} \leq \frac{1.670^2}{2.581^2}\frac{T^3}{f^{1/2}L} \\
\leq \ & \frac{1.670^2c_T^3}{2.581^2}\frac{L^{2-6/r}f^{\frac{3(r-1)}{4r^2}}}{f^{1/2}} \leq 4.265 \frac{L^2}{f^{5/16}} < 1.144 \times 10^{-8} \quad (r=3) \\
u(H_C/H_3)^2 \leq \ & \frac{f}{2.581^2L^4}\frac{q_1}{fh/2} = \frac{2}{2.581^2} \frac{q_1}{L^4h} < \frac{2}{2.581^2} \frac{1}{L^4}  < 4.173 \times 10^{-9}.
\end{align*}
Here, the parenthesised $r$-values maximise the upper bound. Now consider $H_2$ and assume first that $p=5$, so $2\leq r \leq 4$. Thus
\begin{align*}
    u(H_C/H_2)^2 \leq \ & \frac{f}{2.581^2L^4}\frac{q_1}{p^2f^{2-2/p}/(2.1Lq_2)^2} = \frac{2.1^2}{2.581^2}\frac{q_1q_2^2}{L^2p^2f^{1-2/p}} \\
\leq \ & \frac{2.1^2\cdot1.670^2}{2.581^2}\frac{TL}{p^2f^{1/2-2/p}} < \frac{2.1^2\cdot1.670^2\cdot20}{2.581^2\cdot5^2}\frac{rPL}{f^{1/10}} \\
< \ & \frac{2.1^2\cdot1.670^2\cdot20}{2.581^2\cdot5^2} \frac{rL^{1-2/r}}{B_r^{1/r}f^{1/10-\frac{r-1}{4r^2}}} \\
\leq \ & 1.478 \cdot \text{max}\left\{\frac{2}{B_2^{1/2}f^{3/80}},\frac{3L^{1/3}}{B_3^{1/3}f^{4/90}},\frac{4L^{1/2}}{B_4^{1/4}f^{3/64}}\right\} < 0.5498.
\end{align*}
Now, assume instead $p\geq 7$. Then, we have
\begin{align*}
    u(H_C/H_2)^2 \leq \ & \frac{f}{2.581^2L^4}\frac{q_1}{p^2f^{2-2/p}/(2.1Lq_2)^2} = \frac{2.1^2}{2.581^2}\frac{q_1q_2^2}{L^2p^2f^{1-2/p}} \\
\leq \ & \frac{2.1^2\cdot1.670^2}{2.581^2}\frac{TL}{p^2f^{1/2-2/p}} < \frac{2.1^2\cdot1.670^2c_T}{2.581^2\cdot7^2} \frac{L^{2-2/r}}{f^{3/14-\frac{r-1}{4r^2}}} \\
\leq \ & 8.165\times 10^{-2}\frac{L^{2}}{f^{17/112}} \leq 5.876 \times 10^{-4} \quad (r=2).
\end{align*}
In particular, we have
\begin{equation} \label{Eq: CaseII uH_C/H}
    u(H_C/H)^2 \leq \text{max}\left\{u(H_C/H_1)^2,u(H_C/H_2)^2,u(H_C/H_3)^2\right\} < 0.5498.
\end{equation}
Recall that $E_u(X) > 0.7689$ and $E_1(X)<1$. If $q_2\geq T$, we have 
$$\lambda_h(v) = \left\lceil\frac{h}{q_2}\right\rceil < 1 + \frac{h}{T} = 1 + \frac{h-2p}{16rp}. $$
Therefore, we may bound
\begin{equation} \label{Eq: CaseII h-2p}
\left(\frac{h-2p}{h-\lambda_h(q_2)p}\right)^{2r} < \frac{1}{(1-1/16r)^{2r}}<1.136.
\end{equation}
If instead $q_2 < T$, then $q_2\geq h$, leading to $(h-2p)/(h-\lambda_h(q_2)p) = (h-2p)/(h-p) < 1$. Combining the bounds~\eqref{Eq: CaseII-III sig/phi},~\eqref{Eq: CaseII-III E_u(X)},~\eqref{Eq: CaseII uH_C/H} and~\eqref{Eq: CaseII h-2p}, we obtain
\begin{align*}
    \frac{E_1(X_C)}{E_u(X)}\frac{\sigma(u)}{\phi(u)}u\left(\frac{h-2p}{h-\lambda_h(v)p}\right)^{2r}\left(\frac{H_{C}}{H}\right)^2 < \frac{1}{0.7689} \cdot 1.0404 \cdot 1.136 \cdot 0.5498 < 1.
\end{align*}
Since Criterion~\eqref{Eq: proof of main condition} is satisfied, the case follows.

\noindent
{\bf Case III:}  $u=q_1q_2$. Thus, we have $q_1,q_2< T$. We consider each $X_i$ in turn.
\begin{align*}
X_1/u^2 \geq \ & \frac{f}{u^3 h} > \frac{1}{c_T^6}\frac{f^{1-1/2r-\frac{3(r-1)}{2r^2}}}{L^{6-12/r}(1+f^{-1/2r})} \\
\geq \ & \frac{1}{c_T^6}\frac{f^{3/8}}{(1+f^{-1/4})} > 9.662\times 10^{12} \quad (r=2) \\
X_2/u^2 \geq \ & \frac{5f^{4/5}}{2.1 u^2 L q_2 h} > \frac{5}{2.1c_T^5}\frac{f^{4/5-1/2r-\frac{5(r-1)}{4r^2}}}{L^{6-10/r}(1+f^{-1/2r})} \\
\geq \ & \frac{5}{2.1c_T^5}\frac{f^{19/80}}{L(1+f^{-1/4})} > 1.711\times 10^{6} \quad (r=2)\\
X_3/u^2 \geq \ & \frac{f^{1/2}h^{1/2}}{\sqrt{2}u^2h} > \frac{1}{\sqrt{2}c_T^4}\frac{f^{1/2+1/4r-1/2r-\frac{(r-1)}{r^2}}}{L^{4-8/r}(1+f^{-1/2r})} \\
\geq \ & \frac{1}{\sqrt{2}\cdot c_T^4}\frac{f^{1/8}}{(1+f^{-1/2r})}>3207 \quad (r=2).
\end{align*}
In particular, we have $X/u^2 \geq \text{min}\left\{X_1/u^2,X_2/u^2,X_3/u^2\right\} > 3207$. We proceed by bounding each $H_i$ in terms of $H_C$.
\begin{align*}
u(H_C/H_1)^2 \leq \ & \frac{f}{2.581^2L^4}\frac{q_1q_2}{f^2/(q_1^2q_2^2)} = \frac{1}{2.581^2}\frac{q_1^3q_2^3}{fL^4} \\
\leq \ & \frac{c_T^6}{2.581^2}\frac{1}{L^2f^{1/3}} < 8.501\times10^{-17} \quad (r=3) \\
u(H_C/H_2)^2 \leq \ & \frac{f}{2.581^2L^4}\frac{q_1q_2}{p^2f^{2-2/p}/(2.1Lq_2)^2} \leq \frac{2.1^2}{5^2\cdot 2.581^2}\frac{q_1q_2^3}{L^2f^{3/5}} \\
\leq \ & \frac{2.1^2c_T^4}{5^2\cdot 2.581^2}\frac{1}{L^{2/3}f^{7/45}} < 1.716 \times 10^{-8} \quad (r=3) \\
u(H_C/H_3)^2 \leq \ & \frac{f}{2.581^2L^4}\frac{q_1q_2}{fh/2} = \frac{2}{2.581^2}\frac{q_1q_2}{L^4h} \\
< \ & \frac{2c_T^2}{2.581^2}\frac{1}{L^{2}} < 1.662 \times 10^{-4} \quad (r=10).  
\end{align*}
Note that technically $r=10$ is not admissible at $f=10^{40}$ due to the upper bound $r<L/(2\ln 101) \approx 9.978$. This is not an issue since any other value for $r$ would lead to a lower upper bound. In particular, we have
\begin{equation} \label{Eq: CaseIII uH_C/H}
u(H_C/H)^2 \leq \text{max}\left\{u(H_C/H_1)^2,u(H_C/H_2)^2,u(H_C/H_3)^2\right\} \leq 1.662 \times 10^{-4}. 
\end{equation}
Since $v=1$, we have $\lambda_h(v) = 0$, leading to $h-2p<h-\lambda_h(1)p$. Combining the bounds~\eqref{Eq: CaseII-III sig/phi},~\eqref{Eq: CaseII-III E_u(X)} and ~\eqref{Eq: CaseIII uH_C/H}, we obtain
\begin{align*}
    & \frac{E_1(X_C)}{E_u(X)}\frac{\sigma(u)}{\phi(u)}u\left(\frac{h-2p}{h-\lambda_h(1)p}\right)^{2r}\left(\frac{H_{C}}{H}\right)^2 \\
    < \ & \frac{1}{0.7689} \cdot 1.0404 \cdot 1 \cdot 1.662 \times 10^{-4} < 1.
\end{align*}
Since Criterion~\eqref{Eq: proof of main condition} is satisfied, this completes the proof.
\end{proof}
\begin{lemma} \label{L: C_p<D_p}
Fix an integer $2 \leq r < p$. Set $h = \left\lceil f^{1/{2r}}\right\rceil$, $H = f^{1/2}/(2.581L^2)$ and $X = H/h$. If $f\geq 10^{40}$ and $2p<h$, then $C_p(f,h,r)<D_p(f,r)$.
\end{lemma}
\begin{proof}
Since $f\geq 10^{40}$ and 
$$X \geq \frac{f^{1/4}}{2.581L^2(1+f^{-1/4})}\geq 456731, $$
we have
\begin{equation*}
    1 > E_1(X) \geq 1 - \frac{\pi^2}{6}\left(\frac{1}{4}+2+\frac{1}{456731}\right)\frac{1}{456731} \geq 0.9999.
\end{equation*}
By definition, $f^{1/2r}\leq h < f^{1/2r}+1$. Set $y=hf^{-1/2r}$ and define
$$ f_r(y) := yW(f,h,r) = (2r-1)y + r!y^{1-r}.$$
Then $1\leq y < 1+f^{-1/2r}$ and $y^r < (1+f^{-1/2r})^r < e^{rf^{-1/2r}} < e^{1/2}<2$. Whenever $r\geq 3$ is fixed and $y>1$, we have
\begin{align*}
    \frac{f_r(y)-f_r(1)}{y-1} = \ & 2r-1 + r!\frac{y^{1-r}-1}{y-1} \\
    = \ & 2r-1 + r!y^{-1}\frac{1-\left(y^{-1}\right)^{r-1}}{y^{-1}-1} \\
    = \ & 2r-1 - r!\sum_{n=1}^{r-1}\frac{1}{y^n} \\
    \leq \ & 2r-1 - r!\cdot(r-1)\frac{1}{2}.
\end{align*} 
Since $2r-1 - r!\cdot(r-1)/2 < 0$ whenever $r\geq 3$, we have 
$$ y\cdot W(f,h,r) = f_r(y) \leq f_r(1) = 2r-1+r!.$$
If instead $r=2$, then 
$$f_2(y)/f_2(1) = (3y+2y^{-1})/5 \leq 1+ \frac{3}{5} f^{-1/4} \leq 1+6\times 10^{-11}.$$ Therefore, for any $r\geq 2$, we may bound
$$ yW(f,h,r) \leq (1+6\times 10^{-11})(2r-1 + r!), $$
leading to
\begin{align*}
    C_p(f,h,r) = \ & 2.581\left(\frac{1}{E_1(X)}\frac{\pi^2}{6}h\frac{W(f,h,r)}{f^{1/2r}}\right)^{1/2}\left(\frac{h}{h-2p}\right)^r \\
    \leq \ & 2.581\left(\frac{1}{0.9999}\frac{\pi^2}{6}(1+6\times 10^{-11})(2r-1 + r!)\right)^{1/2}\left(\frac{h}{h-2p}\right)^r \\
    < \ & 3.311\frac{(2r-1 +r!)^{1/2}}{(1-2pf^{-1/2r})^{r}}= D_p(f,r),
\end{align*}
as required.
\end{proof}

\begin{proof}[Proof of Theorem~\ref{T: main}]
By Lemma~\ref{L: C_p<D_p}, we have
$$C_p(f,h,r)(p-1)^rL^2 < D_p(f,r)(p-1)^rL^2. $$
If the hypothesis $D_p(f,r)(p-1)^rL^2<f^{\frac{r-1}{4r}}$ holds, then certainly we must also have $C_p(f,h,r)(p-1)^rL^2<f^{\frac{r-1}{4r}}$. Invoking Theorem~\ref{T: main-ugly}, we conclude $K$ cannot be norm-Euclidean.
\end{proof}
\subsection{Computations}
Consider a small prime $p>3$. We will compute optimal pairs $(f_0,r)$ such that $f>f_0$ implies $K$ cannot be norm-Euclidean. We rewrite the Criterion~\eqref{Eq: main} by taking logarithms:
\begin{prop}
Consider a cyclic field $K$ of degree $p>3$ and fix an integer $2 \leq r< p$. Assume $2p < f^{1/2r}$, $f\geq 10^{40}$ and set $L = \ln f$. If
\begin{equation} \label{Eq: comp condition}
r\ln(f^{1/2r}-2p)-\frac{r+1}{4r}L-2\ln L -\ln(3.311)-\frac{1}{2}\ln(2r-1+r!) - r\ln(p-1)> 0,
\end{equation}
then $K$ is not norm-Euclidean.
\end{prop}
For fixed $r$ and $p$, we show that, if some value $f_0\geq 10^{40}$ satisfies the Condition\eqref{Eq: comp condition}, then so does any $f\geq f_0$. Indeed
\begin{align*}
& \frac{\text{d}}{\text{d}f}\left[r\ln(f^{1/2r}-2p)-\frac{r+1}{4r}L-2\ln L -\ln(3.311)-\frac{1}{2}\ln(2r-1+r!) - r\ln(p-1)\right] \\
& = \frac{1}{2f(1-2pf^{-1/2r})}-\frac{r+1}{4rf}-\frac{2}{fL} \geq \frac{1}{f}\left(\frac{1}{2} - \frac{3/2r}{4r} - \frac{2}{105}\right) \geq \frac{89}{840f} > 0
\end{align*}
In particular, the function is increasing in $f$. Thus, we need only find one value $f_0$ for which Criterion~\eqref{Eq: comp condition} holds in order to guarantee it holds for any $f\geq f_0$. For a fixed prime $p>3$, we compute the smallest such admissible $f_0$ for all $2\leq r< p$ and choose the minimal one. This proves Proposition~\ref{P: comp-table}. 

For $p=5$, the optimal choice was $r=4$, whereas $r=3$ was used for all larger primes. The smallest prime for which $r=2$ is optimal is $p=208931$. In fact, one can show that $r=2$ minimises $f_0$ for all primes $p\geq 208931$. Moreover, if we set $r=2$ for any prime, we obtain the simplified form given in Proposition~\ref{P: main-simp}.
\begin{proof}[Proof of Proposition~\ref{P: main-simp}] 
By assumption, we have $2.721 (p-1)  < f^{1/16}L^{-1}$, implying $p<f^{1/16}/2$. Moreover $$f>(2.721 (5-1)L)^{16}> 10^{49}.$$
Thus
$$ D_p(f,2)^{1/2} = 3.311^{1/2}\frac{(2\cdot 2-1+2!)^{1/4}}{1-2pf^{-1/4}} \leq \frac{3.311^{1/2}\cdot 5^{1/4}}{1-(10^{47})^{-3/16}} \leq 2.721. $$
If we set $r=2$, Criterion~\eqref{Eq: main} is equivalent to
$$ D_p(f,2)^{1/2}(p-1) < f^{1/16}L^{-1}. $$
Since 
$$D_p(f,2)^{1/2}(p-1)\leq 2.721 (p-1), \quad f\geq 10^{40} \quad \text{and} \quad 2p < f^{1/4}, $$
we may invoke Theorem~\ref{T: main}. This completes the proof. 
\end{proof}

\section{Future directions} \label{S: future}
We outline some natural extensions to the work presented here. 
\subsection{The constant term in $D_p(f,r)$}
One would expect to gain a substantial improvement for larger $p$ by enforcing $f\geq 10^\alpha$ for some $\alpha>40$. The constant $3.311 \approx 2.581 \frac{\pi}{\sqrt{6}}$ appearing in Criterion~\eqref{Eq: main} is governed by the combined bound $q_1q_2 \leq 2.581 f^{1/2}L^2$. If one chooses a larger cut-off $f\geq 10^\alpha$, then one would expect to immediately obtain an improvement via~\cite[Corollary 7.5]{LezowskiMcGown2017}.

\subsection{Specialising to fixed $p$} It seems reasonable to assume detailed knowledge about the characters of order $p$ would substantially improve the upper bound $f_0$. This is precisely the approach taken in~\cite{BaggerNorm2026} in the case of $p=3$. For example, one should be able to prove an analogue of Vinogradov's trick for cyclic cubic fields: a large $q_1$-value forces $q_2$ to be much smaller than what is initially enforced by Proposition~\ref{P: q1q2}. If the bound $f_0$ can be sufficiently reduced, one could then employ a computational approach to attack the region of unknown from below.

\subsection{Optimising $h = h(f,r)$} Throughout this paper, we set $h=\lceil f^{1/2r}\rceil$ to avoid the $W(f,h,r)$-factor growing unwieldy. Rather, it should be better to instead minimise $h\cdot W(f,h,r) = (2r-1)h + r!f^{1/2}h^{1-r}$. This suggests that one should prefer choosing $h = \lceil k(r) f^{1/2r}\rceil$, for some function $k(r) \approx ((r-1)r!/(2r-1))^{1/r}$.

\end{document}